\documentclass{amsart}
\usepackage{amssymb}
\usepackage[mathscr]{eucal}
\usepackage{eufrak}
\usepackage{xspace}

\newcommand{\R}{\mathbb{R}}
\newcommand{\Q}{\mathbb{Q}}

\newcommand{\N}{\mathbb{N}}

\newcommand{\NN}{\mathcal{N}}
\newcommand{\FF}{\mathcal{F}}

\newcommand{\KK}{\mathcal{K}}

\newcommand{\U}{{\mathcal U}}

\newcommand{\gL}{\Lambda}
\newcommand{\po}{\partial}

\newcommand{\wto}{\rightharpoonup}
\newcommand{\ve}{\varepsilon}

\newcommand{\loc}{{\text{loc}}}
\newcommand{\X}{\times}

\renewcommand{\d}{\delta}
\renewcommand{\l}{\lambda}

\renewcommand{\a}{\alpha}

\newcommand{\s}{\sigma}
\newcommand{\g}{\gamma}

\newcommand{\Om}{\Omega}
\newcommand{\om}{\omega}

\newcommand{\supp}{\text{\rm supp}\,}
\newcommand{\re}{\mathbb{R}}
\newcommand{\M}{{\mathcal M}}

\renewcommand{\S}{{\mathcal S}}
\renewcommand{\subset}{\subseteq}

\newcommand{\K}{\mathcal{K}}

\newcommand{\BUC}{\operatorname{BUC}}

\newcommand{\AP}{\operatorname{AP}}

\newcommand{\BB}{{\mathcal B}}
\renewcommand{\AA}{{\mathcal A}}

\newcommand{\Id}{\operatorname{Id}}

\newcommand{\co}{\operatorname{co}}

\newcommand{\WAP}{\operatorname{{\mathcal W}AP}}
\newcommand{\APs}{\operatorname{{\mathcal W}^*\! AP}}

\newcommand{\oM}{{\operatorname{M}}}

\newcommand{\mm}{{\mathfrak m}}

\newcommand{\mv}[1]{M(#1)}

\newcommand{\medint}{{\mbox{\vrule height3.5pt depth-2.8pt
          width4pt}\mkern-13mu\int\nolimits}}
\newcommand{\Medint}{\mkern12mu\mbox{\vrule height4pt
         depth-3.2pt
          width5pt}\mkern-16.5mu\int\nolimits}

\renewcommand{\supp}{\operatorname{supp}}

\theoremstyle{plain}
\newtheorem{theorem}{Theorem}[section]

\newtheorem{lemma}{Lemma}[section]

\theoremstyle{definition}
\newtheorem{definition}{Definition}[section]
\theoremstyle{remark}

\newtheorem{remark}{Remark}[section]
\numberwithin{equation}{section}

\begin{document}

\title[Weakly* Almost Periodic Homogenization]
{A note on the stochastic weakly* almost periodic\\ homogenization of  fully nonlinear elliptic equations}
\author{Hermano Frid}
\address{Instituto de Matem\'atica Pura e Aplicada - IMPA\\ Estrada Dona Castorina, 110\\
Rio de Janeiro, RJ, 22460-320, Brazil}
\email{hermano@impa.br}

\date{}

\maketitle

\rightline{\textit{Dedicated to Jo\~ao Paulo Dias on his 70th birthday}.\quad\quad\quad\quad }

\begin{abstract} A function $f\in \BUC(\R^d)$ is said to be weakly* almost periodic, denoted $f\in\APs(\R^d)$, if there is $g\in\AP(\R^d)$,  such that, $\oM(|f-g|)=0$, where $\BUC(\R^d)$ and 
$\AP(\R^d)$ are, respectively, the space of bounded uniformly continuous functions and the space of almost periodic functions, in $\R^d$, and $\oM(h)$ denotes the mean value of $h$, if it exists.    We  give a very simple direct proof of the stochastic homogenization property of the Dirichlet problem for fully nonlinear uniformly elliptic equations of the form $F(\om,\frac{x}{\ve},D^2u)=0$, $x\in U$, in a bounded domain $U\subset\R^d$, 
in the case where for almost all $\om\in \Om$, the realization $F(\om,\cdot,M)$ is  a weakly* almost periodic function, for all $M\in\S^d$,  where $\S^d$ is  the space of $d\X d$ symmetric matrices.  Here $(\Om,\mu,\FF)$ is a probability space with probability measure $\mu$  and $\s$-algebra $\FF$ of $\mu$-measurable subsets of $\Om$. For each fixed $M\in\S^d$, $F(\om,y,M)$ is a stationary process, that is, $F(\om,y,M)=\tilde F(T(y)\om,M):= F(T(y)\om,0,M)$, where $T(y):\Om\to\Om$ is an ergodic group of measure preserving mappings such that the mapping  $(\om,y)\to T(y)\om$ is measurable. Also, $F(\om,y,M)$, $M\in\S^d$,  is uniformly elliptic, with ellipticity constants $0<\l<\gL$ independent of $(\om,y)\in\Om\X\R^d$. The result presented here is a particular instance of the general theorem of Caffarelli, Souganidis and Wang, in CPAM 2005. Our point here is just to show a straightforward proof  for this special case, which serves as a motivation for that general theorem, whose proof involves much more intricate arguments.  We remark that  any continuous stationary  process verifies the property that almost all realizations belong to an ergodic algebra, and that $\APs(\R^d)$ is, so far, the greatest known ergodic algebra on $\R^d$. 
\end{abstract}

\section{Introduction} \label{S:0}

We consider the  homogenization of solutions to the following  Dirichlet problem for fully nonlinear elliptic equations 
\begin{align}
F(\om,\frac{x}{\ve},D^2 u)=0, &\quad x\in U \label{e1.1}\\
 u(x)=g(x),&\quad x\in\po U, \label{e1.2}
\end{align}
where $U\subset\R^n$ is a bounded domain with Lipschitz boundary, $g\in C(\po U)$ and $\om\in\Om$, where $(\Om,\mu,\FF)$ is a probability space, with probability measure $\mu$
and $\s$-algebra $\FF$ of $\mu$-measurable sets  $A\subset\Om$.  

Let $\S^d$ denote the space of $d\X d$ symmetric matrices. We assume that, for each $(y,M)\in\R^d\X\S^d$,  $F(\cdot ,y,M):\Om\to\R$ is a bounded $\mu$-measurable function, 
and $F:\Om\X\R^d\X \S^d\to \R$ is uniformly elliptic, that is, for certain constants $0<\l<\Lambda$,
\begin{equation}\label{e1.3}
\l \|N\|\le F(\om,y,M+N)-F(\om,y,M)\le \Lambda \|N\|,
\end{equation}
  for all  $(\om,y)\in \Om\X\R^d$, and all $M,N\in\S^d$, with $N\ge0$.

 We  first recall some well known results for  the deterministic case in which $F$ does not depend on the stochastic parameter $\om$. When, $F(\cdot,M)$ is periodic in $\R^d$,  for each $M\in\S^d$, the solution of this  problem follows from essentially the same ideas on the well know (never to be published) preprint by Lions, Papanicolaou and Varadhan~\cite{LPV} (see also \cite{Ev}).  Also, when $F(\cdot, M)$ is almost period in $\R^d$, the homogenization problem can be addressed using the same ideas as those employed by Ishii in \cite{I}. The adaptations of the ideas in \cite{I} for the case of fully nonlinear elliptic equations were described in \cite{CSW} and also reviewed in \cite{AF}. It is important to remark that, while the uniform convergence of the solutions of the Hamilton-Jacobi equations, both in \cite{LPV} and \cite{I}, is guaranteed by the uniform boundedness of $Du^\ve$, which follows from  the coerciveness of the nonlinearity, in the case of the fully nonlinear elliptic equations, the compactness of the solutions $u^\ve$ is obtained through an application of the H\"older continuous regularity theorem by Caffarelli~\cite{Ca} (see also \cite{CC}).  
 
As for the determination of the effective nonlinearity,  roughly speaking, we have the following. In the periodic case, the compactness of the torus $\Pi$ allows the solution of the corrector (or cell) equation 
 $$
 F(y,M+D_y^2 v)=\l,\quad y\in\Pi,
 $$
 for the definition of the effective Hamiltonian, $\bar F(M)$,  as the corresponding unique ``nonlinear eigenvalue'' $\l$ for which this equation possesses a solution. On the other hand,  
 in the almost periodic case,  since the domain of the corrector equation is now  the whole $\R^d$, following Ishii~\cite{I}, that equation is replaced by a family of $\d$-approximate corrector equations
 $$
\d v+  F(y,M+D_y^2 v)=0,\quad y\in\R^d.
 $$
 A compactness criterion for families  of almost periodic functions  is then used to guarantee the uniform convergence of $\d$ times the solutions $v_\d$ of the $\d$-approximate corrector equations, and this limit, which is proven to be unique and constant, depending only on $M$, is the  effective nonlinearity $\bar F(M)$ for the almost periodic homogenization problem.  

In~\cite{CSW}, Caffarelli, Souganidis and Wang solved the stochastic homogenization problem for \eqref{e1.1}, in the most general form where $F$ may also depend explicitly on $x,u, Du$,   and, for each $M$, $F(\om,y,M)$ is a {\em stationary ergodic process}. The latter means that 
\begin{equation}\label{e1.sto}
F(\om,y,M)=\tilde F(T(y)\om,M), 
\end{equation}
with $\tilde F(\cdot,M):=F(\cdot,0,M)$, for $T(y)\,:\, \Om\to\Om$, $y\in\R^d$, satisfying: 
\begin{enumerate}
\item[(T1)]{(\textsc{Group property})}  $T(0)=\Id$, $T(y+z)=T(y)\circ T(z)$, where $\Id\,:\, \Om\to\Om$ is the identity mapping;  
\item[(T2)]{(\textsc{Measurability})} The mapping $T\,:\,\R^d\X\Om\to\Om$, defined by $T(y,\om)=T(y)\om$, is measurable in the product $\s$-algebra $\BB^d\X \FF$, where $\BB^d$ denotes the Borel $\s$-algebra in $\R^d$; 
\item[(T3)]{(\textsc{Invariance of $\mu$})} $\mu(T(y) A)= \mu(A)$, for all $A\in\FF$;
\item[(T4)]{(\textsc{Ergodicity})} If $A\in\FF$ satisfies $T(y)A=A $, for all $y\in\R^d$, then either $\mu(A)=0$ or $\mu(A)=1$.  
\end{enumerate}

In the recent paper \cite{AS}, Armstrong and Smart simplify the proof in \cite{CSW} and improve the treatment of the case in which $F$ depends also on the gradient $Du$.

The main purpose of this Note is to give a very simple and direct proof of the result in \cite{CSW} under the {\em additional assumption} that almost all realization $F(\om,\cdot,M)$  is a weakly* almost periodic function, that is,
\begin{equation}\label{e1.F}
F(\om,\cdot,M) \in \APs(\R^d), \quad \text{for all $M\in\S^d$, $\mu$-a.a.\ $\om\in\Om$},
\end{equation}
where we denote by $\APs(\R^d)$  the space of so called weakly* almost periodic functions, defined as the algebraic sum 
$$
\APs(\R^d):=\AP(\R^d)+\NN(\R^d),
$$ 
where $\AP(\R^d)$ denotes the space of almost periodic functions and $\NN(\R^d)$ is the subspace of the bounded uniformly continuous functions $\BUC(\R^d)$, defined by  
\begin{equation}\label{eN}
\NN:= \{f\in \BUC(\R^d)\, :\, \lim_{R\to\infty}\frac{1}{|B(0;R)|}\int_{B(0;R)}|f(y)|\,dy=0\},
\end{equation}
where, as usual, $B(0;R)$ denotes the ball centered at the origin with radius $R$, and $|B|$ denotes the $d$-dimensional Lebesgue measure of $B$. 
We believe that this special case serves, on one hand, as a good verification of the general result of \cite{CSW}, in which it is possible a much simpler definition for the effective Hamiltonian. On the other hand, it also serves as a good instance for the application of results in the abstract theory of ergodic algebras.

  In this Note, we give a simple direct proof of the following result, which is a particular case of the general result in \cite{CSW}. We still need one more condition on $F$, which is also required in \cite{CSW} and \cite{AS} and is needed in connection with the basic results in the theory of viscosity solutions of fully nonlinear elliptic equations, see \cite{CIL}. So, as in \cite{AS},  we assume that there exists a   function $\rho:[0,\infty)\to [0,\infty)$ satisfying 
  $\lim_{s\to0+}\rho(s)=0$, and $\g>\frac12$, such that,  for all $(\om,M)\in\Om\X\S^d$, $y,z\in\R^d$,  
 \begin{equation}\label{eF3}
|F(\om,y,M)-F(\om,z,M)|\le \rho((1+|M|)|y-z|^\g).
 \end{equation}
  
 \begin{theorem}[{\em cf.} \cite{CSW}] \label{T:1.1}  Let $F(\s,y,M)$ satisfy \eqref{e1.3}, \eqref{e1.sto}, \eqref{e1.F}, \eqref{eF3},  with $T(y)$ satisfying (T1)--(T4). Then, there exists $\bar F:\S^d\to\R$ satisfying the uniform ellipticity condition
 \begin{equation}\label{e1.T1}
 \l\|N\|\le \bar F(M+N)-\bar F(M)\le \Lambda \|N\|,
 \end{equation}
  with $\l,\Lambda$ as in \eqref{e1.3}, such that, for any bounded open set $U$, with Lipschitz boundary $\po U$, and $g\in C(\po U)$, $\mu$-almost surely,  the solutions  of \eqref{e1.1}-\eqref{e1.2},  $u^\ve(\om,\cdot)$,   converge uniformly to the unique viscosity solution $u$ of
  \begin{align}
\bar F(D^2 u)=0, &\quad x\in U \label{e1.1h}\\
 u(x)=g(x),&\quad x\in\po U \label{e1.2h}
\end{align}
  
  \end{theorem}

  We remark that the only fact used to prove Theorem~\ref{T:1.1}, related specifically with problem, \eqref{e1.1}-\eqref{e1.2}, is  that the corresponding deterministic almost periodic homogenization problem has been solved, besides an usual hypothesis of equicontinuity on the microscopic variable,  such as \eqref{eF3}, in general also needed for solving the equation itself and for finding the effective nonlinearity. So, in particular, instead of \eqref{e1.1} we could have a much more general fully nonlinear uniformly elliptic or parabolic equation, or Hamilton-Jacobi, etc. In sum, the point we make here is a very general one. It basically means that adding the condition that almost surely the realizations belong to $\APs(\R^d)$, reduces any stochastic homogenization problem, under an usual equicontinuity assumption on the microscopic variable, to the corresponding deterministic almost periodic homogenization problem.

  The proof of Theorem~\ref{T:1.1} is given in the following Section~\ref{S:2}.   We also provide in the Appendix~\ref{A:1} a brief exposition of the subject of ergodic algebras, including the statement of Theorem~\ref{T:1.5},  asserted in \cite{JKO}  with detailed proof in \cite{AFS}, which establishes that almost all realizations of continuous functions by a stationary process belong to an ergodic algebra. This point is important  since $\APs(\R^d)$ is the greatest so far known ergodic algebra in $\R^d$. In Appendix~\ref{A:2},  we provided a detailed discussion about  the weakly almost periodic functions, introduced by Eberlein in \cite{E1}.

\section{Proof of Theorem~\ref{T:1.1}.}\label{S:2}

{\em Step 1}.  We may assume that $\FF$ is the $\s$-algebra generated by the countable family of sets $\{\om\in\Om\,:\, a<F(\om,y,M)<b\,\}$, with $a,b\in\Q$, $y\in\Q^d$, $M\in\S^q\cap\Q^{d^2}$. Let us define $\U_\FF$ as the closed algebra with unity  generated by the functions  $\tilde F(T(y)\,\cdot ,M)$, for $y\in\Q^d$ and $M\in\S^d\cap\Q^{d^2}$. We may assume that the algebra  $\U_\FF$  distinguishes between the points of $\Om$, that is, given any two distinct points 
$\om_1,\om_2\in\Om$, there exists $g\in\U_\FF$ such that $g(\om_1)\ne g(\om_2)$; if this is not true we may replace $\Om$ by its quotient by a trivial equivalence relation, in a standard way, and we proceed correspondingly with $\FF$ and $\mu$.  Thus, using a classical theorem of Stone (see, e.g., \cite{DS}, theorem 20, p.276), we may embed $\Om$ densely in a compact (separable) topological space $\tilde \Om$ so that $\U_\FF$ is isometrically identified to $C(\tilde \Om)$.  In this way, $\FF$ extends to the collection, $\tilde \FF$, of Borel sets of $\tilde \Om$ and $\mu$ extends to a probability measure $\tilde\mu$ on $\tilde \Om$ by setting $\tilde \mu(A)=\mu(A\cap \Om)$, for all $A\in\tilde\FF$. It is also a standard matter (see, e.g., \cite{AFS}) to check that the mappings $T(y):\Om\to\Om$ extend to a group of homeomorphisms $\tilde T(y): \tilde\Om\to \tilde\Om$, satisfying (T1)--(T4) with $\Om$ replaced by $\tilde\Om$, but now, using also \eqref{eF3}, we have that the mapping $\tilde T:\R^d\X \tilde\Om\to \tilde\Om$, $\tilde T(y,\tilde \om)=\tilde T(y)\tilde \om$
is continuous.  Indeed, given any $g\in\U_\FF$, $g$ is identified as a representative of $C(\tilde\Om)$, and, if $(\om_k,y_k)\to(\om,y)$, 
\begin{multline*}
|g(T(y_k)\om_k)-g(T(y)\om)|\le |g(T(y_k)\om_k)-g(T(y)\om_k)|+|g(T(y)\om_k)-g(T(y)\om)|\\
\le O(|y_k-y|)+|g(T(y)\om_k)-g(T(y)\om)|\to 0,\quad\text{as $k\to\infty$},
\end{multline*}
where we used \eqref{eF3} and the continuity of $g(T(y)\,\cdot)$, for each fixed $y\in\R^d$.

Moreover, clearly, $\tilde \mu$ is an invariant measure for $\tilde T(y)$ , $y\in\R^d$. The family $\{(\tilde T(y),\tilde\Om,\tilde \mu)\,:\, y\in\R^d\}$ constitutes a so called continuous dynamical system, which is ergodic by (T4). In what follows in this proof, we will no longer use the ``$\tilde{\phantom{h}}$'' to distinguish the extensions from the original objects, that is, we assume that $\{(T(y), \Om, \mu)\,:\,y\in\R^d\}$ is a  continuous ergodic dynamical system on the separable compact space $\Om$. 

{\em Step 2}.    We claim that there exists $F_*:\Om\X\R^d\X\S^d\to\R$, such that, for each $(y,M)\in\R^d\X\S^d$, $F_*(\cdot,y,M)$ is $\FF$-measurable, and for, $\mu$-a.a.\ $\om\in\Om$ and all $M\in\S^d$, $F_*(\om,\cdot,M)$ is the almost periodic component of $F(\om, \cdot, M)$, that is,
\begin{equation}\label{e2.1}
F_*(\om,\cdot,M)\in\AP(\R^d),\quad  \lim_{R\to\infty}\frac{1}{|B(0;R)|}\int_{B(0;R)}|F(\om,y,M)-F_*(\om,y,M)|\,dy=0.
\end{equation}
Moreover, if $\tilde F_*(\om,M):=F_*(\om,0,M)$, for all $(\om,M)\in\Om\X\S^d$, then $F_*(\om,y,M)=\tilde F_*(T(y)\om,M)$, with $T$ given by \eqref{e1.sto} satisfying (T1)--(T4). 

Indeed, by assumption $F(\om,\cdot,M)\in \APs(\R^n)$, for $\mu$-a.a.\ $\om\in\Om$ and all $M\in\S^d$, and by definition of $\APs(\R^d)$, we have
\begin{equation}\label{e2.2}
F(\om,\cdot,M)=F_*(\om,\cdot,M)+ F_\NN(\om,\cdot,M),
\end{equation}
for certain $F_*(\om,\cdot,M)\in\AP(\R^d)$ and $F_\NN(\om,\cdot,M)\in\NN(\R^d)$.  Let us assume that \eqref{e2.2} holds for $\om\in\Om_*$, with $\mu(\Om_*)=1$. Since we may find
compact sets $K_n\subset\Om_*$, with $\mu(K_n)\ge 1-\frac1{n}$, for all $n\in\N$, we may assume that \eqref{e2.2} holds everywhere in the compact space $\Om$.  

Now, by the properties of $\AP(\R^d)$ (see, e.g., \cite{DS}) it is well known that there exists in 
$\AP(\R^d)$ a generalized sequence of approximations of the unity $\{\phi_\a\}_{\a\in D}$, for some directed set $D$, such that, for all $\varphi\in\AP(\R^d)$,  
$$
\phi_\a{*}_{{}_\oM}\varphi\underset{\a\in D}{\to} \varphi,\quad\text{where}\quad \phi_\a{*}_{{}_\oM}\varphi(y):=\oM(\phi_\a(\cdot)\varphi(y-\cdot))=\lim_{R\to\infty}\frac{1}{|B(0;R)|}\int_{B(0;R)}\phi_\a(z)\varphi(y-z)\,dz,
$$  
where the convergence is in the $\sup$-norm in $\R^d$, and we also assume, as we may, that $\phi_\a(-y)=\phi_\a(y)$, for all $y\in\R^d$, for all $\a\in D$.  

Since, clearly, $\phi_\a *_{{}_\oM}F_\NN(\om,\cdot,M)\equiv0$, for all $\om\in\Om$, we have that
$$
\phi_\a*_{{}_\oM}F(\om,\cdot,M)\underset{\a\in D}{\to}F_*(\om,\cdot,M),\quad\text{for all  $\om\in\Om$ and all $M\in\S^d$}.
$$
We remark that the above convergence is uniform and we may replace $\phi_\a$ by a  sequence $\phi_n$. We also remark that \eqref{eF3} allows to choose the sequence  $\phi_n$ 
independently of $\om\in\Om$, or $M\in\S^d$, which can be seen as follows.

Indeed, since $\Om$ is separable, we may endow it  with a metric. Using \eqref{eF3} and the compactness of the metric space $\Om$,
we choose a sequence $\phi_n$ such that $\phi_n*_{{}_\oM}F(\om,\cdot,M)$  converges locally uniformly in $\Om\X\R^d$, and  $\phi_n*_{{}_\oM}F(\om,\cdot,M)=\phi_n*_{{}_\oM}F_*(\om,\cdot,M)$, for all $\om\in\Om$.  In fact, for each $N\in\N$, using \eqref{eF3}, we may define a subsequence of  $\phi_\a*_{{}_\oM}F(\om,\cdot,M)$ converging, uniformly for $|y|\le N$ and $\om$ in a countable dense subset  $D\subset \Om$, to a function $G_N(\cdot,\cdot,M)$, defined in $D\X\{|y|\le N\}$. Now the uniform continuity of $F(\cdot,\cdot,M)$ on
$\Om\X\{|y|\le N\}$, implies the uniform continuity of the limit $G_N(\cdot,\cdot,M)$ in $D\X\{|y|\le N\}$ and so it can be extended continuously in a unique way to $\Om\X\{|y|\le N\}$.
We then prove, in a standard fashion, that the sequence so defined converges to $G_N(\cdot,\cdot,M)$ uniformly in $\Om\X\{|y|\le N\}$. By taking successive subsequences for $N=1,2,\cdots$, by a diagonal process we may then define the locally uniformly convergent sequence    $\phi_n*_{{}_\oM}F(\om,\cdot,M)$, to some continuous function $G(\cdot,\cdot,M)$, in $\Om\X\R^d$.

Now, for each $\om\in\Om$, $\phi_n*_{{}_\oM}F_*(\om,\cdot,M)$ is compact in the $\sup$-norm.  Indeed, each $\phi_n*_{{}_\oM}F_*(\om,\cdot,M)$  belongs to $\overline{\co(O(F_*(\om,\cdot,M))}$, where $O(F_*(\om,\cdot,M)):=\{ F_*(\om,\cdot+\l,M)\,:\, \l\in\R^d\}$, which is pre-compact in $\AP(\R^d)$ by Bochner's criterion, and  $\overline{\co(O(F_*(\om,\cdot,M))}$ denotes the closure of the convex hull of  $O(F_*(\om,\cdot,M))$,  which is then also compact in the $\sup$-norm, by a well known result (see, e.g.,  \cite{Ru0}, p.72).
Thus, we may extract a subsequence converging uniformly to an almost periodic function,
which then must coincide with $F_*(\om,\cdot,M)$. Hence, we conclude that $\phi_n*_{{}_\oM}F(\cdot,\cdot,M)$ converges locally uniformly to $F_*(\cdot,\cdot,M)$.

The $\FF$-measurability of $F_*(\om,y,M)$ is then clear in view of these observations. It is also immediate to check that $F_*(\om,y,M)=\tilde F_*(T(y)\om,M)$,  with $T$ given by \eqref{e1.sto} satisfying (T1)--(T4).

{\em Step 3}. We claim that  
\begin{equation}\label{e2.3}
F(\om,y,M)=F_*(\om,y,M),\qquad \text{ for $\mu$-a.a.\ $\om\in\Om$, for all $(y,M)\in\R^d\X\S^d$}.
\end{equation}
 Indeed, for each $(y,M)\in \R^d\X\S^d$, we have
 \begin{equation}\label{e2.4}
 \begin{aligned}
 &\int_{\Om}|\tilde F(T(y)\om,M)-\tilde F_*(T(y)\om,M)|\,d\mu(\om)=\int_{\Om}|\tilde F(\om, M)-\tilde F_*(\om,M)|\,d\mu(\om)\\
 &\qquad\qquad=\oM\left(|\tilde F(T(\cdot)\om_*,M)-\tilde F_*(T(\cdot)\om_*,M)|\right)=0,\quad \text{for $\mu$-a.a.\ $\om_*\in\Om$},
 \end{aligned}
 \end{equation}
 by the invariance of $\mu$ with respect to $T(y)$, Birkhoff's relation and \eqref{e1.F}. Thus, \eqref{e2.4} implies that we may find a subset $\Om_*\subset\Om$, with $\mu(\Om_*)=1$, such that 
 $$
 \tilde F(T(y)\om, M)=\tilde F_*(T(y)\om,M), \quad\text{for all $\om\in\Om_*$,  $y\in\Q^d$, $M\in\S^d\cap\Q^{d^2}$},  
 $$
 which, by density and continuity, gives \eqref{e2.3}.

{\em Step 4} (Homogenization). By \eqref{e2.3}, it is the enough to consider the stochastic almost periodic homogenization problem, that is, in \eqref{e1.1}-\eqref{e1.2}, it suffices to consider the case where, instead of \eqref{e1.F}, we have   
\begin{equation}\label{e1.F*}
F(\om,\cdot,M) \in \AP(\R^d), \quad \text{for all $M\in\S$, $\mu$-a.a.\ $\om\in\Om$},
\end{equation}
$\Om$ is a separable compact space  and $\{(T(y),\Om,\mu)\,:\, y\in\R^d\}$ is a continuous ergodic $d$-dimensional dynamical system. 

{\em \#4.1}.  So, we fix $\om\in\Om$ verifying \eqref{e1.F*},  and solve the homogenization problem for   \eqref{e1.1}-\eqref{e1.2} using the blueprint described in \cite{CSW} and detailed in \cite{AF}, which builds upon the ideas from Ishii~\cite{I}. Here, the existence and uniqueness of viscosity solutions to \eqref{e1.1}-\eqref{e1.2} follow from the general theory presented in \cite{CIL} and the references therein. 

The compactness of the sequence of solutions $u^\ve$ of \eqref{e1.1}-\eqref{e1.2} follows from the uniform Holder continuity estimate of Caffarelli~\cite{Ca} (see also \cite{CC}).   The approximate 
$\d$-corrector equation is then introduced, as mentioned in the Introduction, 
 \begin{equation}\label{e2.5}
\d v+  F(\om,y,M+D_y^2 v)=0,\quad y\in\R^d.
 \end{equation}
 The compactness of $\d v_\d$, where $v_\d$ is the solution of \eqref{e2.5}, follows from an application of a classical general criterion for the pre-compactness of a family of almost periodic functions (see, e.g., \cite{AF}). The effective operator $\bar F(\om,M)$ is then defined as the limit $\d v_\d$, in the $\sup$-norm, which is easily seen to be constant, depending on $M$ and, in principle, also on $\om$.  A routine procedure shows that $\bar F(\om,M)$ satisfies the uniform ellipticity condition \eqref{e1.3}, with the same constants $\l,\Lambda$.
 It is then shown that any uniform limit of a subsequence of the solutions $u^\ve$, of \eqref{e1.1}-\eqref{e1.1}, is the unique viscosity solution $u$ of 
 \begin{align}
 &\bar F(\om, D^2 u)=0, \quad x\in U, \label{e2.6}\\
 & u|\po U= g. \label{e2.7}
 \end{align}

{\em \#4.2}. It remains to prove that $\bar F(\om, M)$ does not depend on $\om$. This follows from the ergodicity of $T(y)$, $y\in\R^d$, which so far has not been used.   So, we claim that $\bar F(\om,M)$ is an invariant function, for each $M\in\S^d$.  

Indeed, if $\om\in\Om$ verifies \eqref{e1.F*}, then let us take an arbitrary $y_0\in\R^d$. We prove that $\bar F(T(y_0)\om,M)=\bar F(\om,M)$, for all $M\in\S^d$.  In fact, for any $\d>0$,
if $v_\d(\om, y, M)$ is the solution of \eqref{e2.5} and $v_\d(T(y_0)\om,y,M)$ is the solution of the same equation with $\om$ replaced by $T(y_0)\om$,  using \eqref{e1.sto}, it is easy to see that 
$$
v_\d(T(y_0)\om, y, M)=v_\d(\om, y+y_0,M),\quad \text{for all $y\in\R^d$}.
$$
Since $\bar F(T(y_0)\om,M)$ is the uniform limit in $\R^d$, as $\d\to0$,  of $\d v_\d(T(y_0)\om, \cdot,M)$, we conclude that $\bar F(T(y_0)\om,M)=\bar F(\om,M)$, for all $\om$ verifying \eqref{e1.F*}.
Now, by the ergodicity of $T(y)$, we conclude that $\bar F(\om,M)$ does not depend on $\om$, and so $\bar F(\om,M)=\bar F(M)$, which implies that for a.a.\ $\om\in\Om$ the solutions
$u^\ve(\om,x)$ converge as $\ve\to0$, uniformly to the unique viscosity solution of   
\begin{align}
 &\bar F(D^2 u)=0, \quad x\in U, \label{e2.8}\\
 & u|\po U= g, \label{e2.9}
 \end{align}
which finishes the proof of Theorem~\ref{T:1.1}.

\appendix
\section{Algebras with mean value.}\label{A:1}

In this section we recall the basic facts concerning algebras with mean values and, in particular, ergodic algebras. To begin with, we recall
the notion of mean value for functions defined in $\re^d$.

\begin{definition}\label{D:3} Let $g\in L_\loc^1(\R^d)$. A number $\oM(g)$ is called the {\em mean value of $g$} if
\begin{equation}\label{eA.2}
\lim_{\ve \to0} \int_Ag(\ve^{-1}x)\,dx=|A|\oM(g)
\end{equation}
for any Lebesgue measurable bounded set $A\subset\R^d$, where $|A|$ stands for the Lebesgue measure of $A$.
This is equivalent to say that $g(\ve^{-1}x)$ converges, in the
duality with $L^\infty$ and compactly supported functions, to the constant $\mv{g}$.
Also,
if $A_t:=\{x\in\R^n\,:\, t^{-1}x\in A\}$ for $t>0$ and $|A|\ne0$, \eqref{e1.2} may be written as
\begin{equation}\label{eA.3}
\lim_{t\to\infty}\frac1{t^d|A|}\int_{A_t}g(x)\,dx=\oM(g).
\end{equation}
\end{definition}
Also, we will use the notation $\medint_{A}g\,dx$ for the average or mean value of $g$ on the measurable set $A$, and $\medint_{\R^d}g\,dx$ or  $\overline{g}$ for $\oM(g)$, given by \eqref{e1.3}.

We recall now the definition of algebras with mean value introduced in \cite{ZK}. As usual, we denote by $\BUC(\R^d)$ the
space of the bounded uniformly continuous real-valued functions in $\R^d$.

\begin{definition}\label{D:5} Let $\AA$ be a linear subspace of $\BUC(\R^d)$.
We say that $\AA$ is an {\em algebra with mean value} (or {\em
algebra w.m.v.}, in short), if the following conditions are
satisfied:
\begin{enumerate}
\item[(A)] If $f$ and $g$ belong to $\AA$, then the product $fg$ belongs to $\AA$.
\item[(B)] $\AA$ is invariant with respect to translations $\tau_y$ in $\R^d$.
\item[(C)] Any $f\in\AA$ possesses a mean value.
\item[(D)] $\AA$ is closed in $\BUC(\R^d)$ and contains the unity, i.e., the function $e(x):=1$ for $x\in\R^d$.
\end{enumerate}
\end{definition}


For the development of the homogenization theory in algebras
with mean value, as is done in \cite{ZK,JKO} (see also \cite{CG}),
in similarity with the case of almost periodic functions, one
introduces, for $1\leq p<\infty$, the space  $\BB^p$ as the abstract
completion of the algebra
$\AA$ with respect to the Besicovitch seminorm
$$
|f|_p^p:=\limsup_{L\to\infty}\frac{1}{(2L)^d}\int_{[-L,L]^d}|f|^p\,dx.
$$
Both the action of translations and the mean value
extend by continuity to $\BB^p$, and we will keep using the notation
$f(\cdot+y)$ and $\oM(f)$ even when $f\in\BB^p$ and $y\in\R^d$. Furthermore,
for $p>1$ the product in the algebra extends to a bilinear operator from $\BB^p\times\BB^q$ into $\BB^1$,
with $q$ equal to the dual exponent of $p$, satisfying
$$
|fg|_1\leq |f|_p|g|_q.
$$
In particular, the operator $\oM(fg)$ provides a nonnegative definite
bilinear form on $\BB^2$.

Obviously the corresponding quotient spaces for all these spaces
(with respect to the null space of the seminorms) are Banach spaces,
and we get a Hilbert space in the case $p=2$. We denote by
$\overset{\BB^p}{=}$, the equivalence relation given by the equality
in the sense of the $\BB^p$ semi-norm.

\begin{remark}\label{R:0.1} A classical argument going back to Besicovitch~\cite{B} (see also \cite{JKO}, p.239) shows that the elements of $\BB^p$ can be represented
by functions in $L_{\loc}^p(\R^d)$, $1\le p<\infty$.
\end{remark}

We next recall a result established in \cite{AFS} which provides a connection between algebras with mean value and compactifications  of $\R^d$ endowed with a group of ``translations'' and an invariant probability measure.

\begin{theorem}[cf.\ \cite{AFS}]\label{T1}
For an algebra w.m.v.\ $\AA$, we have:
\begin{enumerate} 
\item[(i)] There exist a compact space
${\mathcal K}$ and an isometric isomorphism $i$ identifying $\AA$ with the
algebra $C({\mathcal K})$ of continuous functions on ${\mathcal K}$. By abuse of notation we will make the identification $i(f)\equiv f$, for all $f\in\AA$. When $\AA$ distinguishes between points, there is a canonical embedding $\R^d\subset \KK$, with dense image. 
\medskip

\item[(ii)] The translations $T(y):\R^d\to\R^d$, $T(y)x=x+y$,
extend to a group of homeomorphisms $T(y):{\mathcal K}\to{\mathcal K}$, $y\in\R^d$, i.e., $T(0)=\Id$, $T(y_1+y_2)=T(y_1)\circ T(y_2)$. The map $T:\R^n \X\KK\to\KK$, given by $T(y,z):=T(y)z$ is continuous. In other words, $T(y)$, $y\in\R^d$, is a continuous ($d$-dimensional) dynamical system over $\KK$. 
\medskip

\item[(iii)] There exists a Radon probability measure ${\mathfrak m}$ on
${\mathcal K}$ which is invariant by the
group of transformations $T(y)$, $y\in\R^d$, such that
$$
\Medint_{\R^d}f\,dx=\int_{\mathcal K} f\,d\mathfrak m.
$$
\medskip
\item[(iv)]For $1\le p\le \infty$, the
Besicovitch space $\BB^p\big/\overset{\BB^p}{=}$ is
isometrically isomorphic to $L^p({\mathcal K}, {\mathfrak m})$.
\end{enumerate}
Actually, {\rm(i)} and {\rm(ii)}  hold independently of the mean value property {\rm(C)} in the definition of algebra w.m.v. 
\end{theorem}

A group of unitary operators
$T(y):\BB^2\to\BB^2$ is then defined by setting $[T(y)f](\cdot) := f(T(y,\cdot))$. Since the elements of $\AA$ are uniformly continuous in $\R^d$,
the group $\{T(y)\}$ is strongly continuous, i.e. $T(y)f\to f$
in $\BB^2$ as $y\to 0$ for all $f\in\BB^2$. The notion of invariant
function is introduced then by simply
saying that a function in $\BB^2$ is {\em invariant} if
$T(y)f\overset{\BB^2}{=} f$, for all $y\in\R^n$. More clearly,
$f\in\BB^2$ is invariant if
\begin{equation}\label{e1.INV}
\oM\bigl(|T(y)f-f|^2\bigr)=0,\qquad \forall y\in\R^n.
\end{equation}
The concept of ergodic algebra is then introduced as follows.

\begin{definition}\label{D:6} An algebra $\AA$ w.m.v.\ is called {\em ergodic} if any invariant function
$f$ belonging to the corresponding space $\BB^2$ is equivalent (in $\BB^2$) to a constant.
\end{definition}

In \cite{JKO}  it is also given an alternative definition of ergodic
algebra which is shown therein to be equivalent to
Definition~\ref{D:6}, by using  von~Neumann mean ergodic theorem.
We state that as the following lemma, whose detailed proof may be
found in \cite{JKO}, p.247.

\begin{lemma}\label{L:1.6} Let $\AA$
be an algebra with mean value on $\R^d$. Then $\AA$ is ergodic
if and only if
\begin{equation}\label{eL1.6}
\lim_{t\to\infty}M_y\left(\bigl|\frac{1}{|B(0;t)|}
\int_{B(0;t)}f(x+y)\,dx-M(f)\bigr|^2\right)=0
\qquad\forall f\in\AA.
\end{equation}
\end{lemma}

Finally, the following theorem stated in \cite{JKO}, whose detailed proof is given in \cite{AFS}, is a sort of converse of Theorem~\ref{T1}.

\begin{theorem}[{\em cf.} \cite{AFS}] \label{T:1.5}  Let ${\mathcal Q}$ be a compact space, $T(x):{\mathcal Q}\to {\mathcal Q}$, $x\in\R^d$, a continuous
dynamical system, $\mu$ a Radon probability invariant measure in ${\mathcal Q}$, and
$V\subset C({\mathcal Q})$ a separable subspace. Then, for $\mu$-almost all $\om\in{\mathcal Q}$, there is an ergodic algebra containing the set of realizations 
$\big{\{}f(T(\cdot)\omega)\big{\}}_{f\in V}$ .
\end{theorem}
 
We finally remark that $\APs(\R^d)$ is an ergodic algebra. Indeed, since $\APs(\R^d)$ is defined simply as the algebraic sum of $\AP(\R^d)$ and $\NN(\R^d)$, and since, for any $g\in \BUC(\R^d)$ and $h\in\NN(\R^d)$, $gh\in \NN(\R^d)$,  we immediately check that $\APs(\R^d)$ is a sub algebra  of  $\BUC(\R^d)$ and it is trivially invariant by translations. It is also immediate that any element of $\APs(\R^d)$ possesses a mean value. So, the only property to be checked is that $\APs(\R^d)$ is a closed subspace of $\BUC(\R^d)$. This is indeed a consequence of the fact that both $\AP(\R^d)$ and $\NN(\R^d)$ are closed subspaces of $\BUC(\R^d)$ such that $\AP(\R^d)\cap\NN(\R^d)=\{0\}$, and so we can apply the closed graph theorem to the mapping $\AP(\R^d)\X\NN(\R^d)\to \APs(\R^d)$, $(g,h)\mapsto g+h$.

\section{Weakly Almost Periodic Functions} \label{A:2}

Examples of ergodic algebras include the periodic continuous functions, the almost periodic functions, and the Fourier-Stieltjes transforms, studied in \cite{FS}. 
More generally, all the just mentioned ergodic algebras are subalgebras of a strictly larger ergodic algebra,  that is the algebra of the (real-valued) {\em weakly almost periodic functions} in $\R^d$, $\WAP(\R^d)$. It is defined as the subspace of the space of the bounded continuous functions,  $C(\R^d)$, formed by those  $f: \R^d\to\R$, satisfying the property that any sequence of its translates $(f(\cdot+\l_i))_{i\in\N}$ possesses a subsequence $(f(\cdot+\l_{i_k}))_{k\in\N}$ weakly converging in $C(\R^d)$.
This space was introduced and its main properties were obtained by Eberlein in \cite{E1} (see also \cite{E2}). In particular, in \cite{E1}, Eberlein proved that $\WAP(\R^d)$ satisfies all the properties defining an algebra w.m.v. It is immediate to see, from the definition, that $\operatorname{WAP}(\R^d)\supset \AP(\R^d)$,  where the latter denotes the space of almost periodic functions. Indeed, for functions in $\AP(\R^d)$, Bochner theorem gives the relative compactness of the translates $f(\cdot+\l)$, $\l\in\R^n$,  in the $\sup$-norm (see, e.g., \cite{B}). 

We summarize in the following lemma the properties of $\WAP(\R^d)$ which were essentially proved by Eberlein in   \cite{E1}. The proof below is borrowed from \cite{FSV} for the reader's convenience. 

\begin{lemma}[{\em cf.}  \cite{E1}] \label{L:2.wap} $\WAP(\R^d)$ is an ergodic algebra which contains the algebra of Fourier-Stieltjes transforms 
$\operatorname{FS}(\R^d)$. 
\end{lemma}

\begin{proof} The fact that  $\WAP(\R^d)\subset \BUC(\R^d)$ is proved by contradiction. Assume, on the contrary, that one can find points $\xi_k,\s_k$, with $|\xi_k-\s_k|\to0$ as $k\to\infty$, such that $|f(\xi_k)-f(\s_k)|\ge\ve_0>0$, for all $k\in\N$. Define $g_k(x)=f(x+\xi_k)-f(x+\s_k)$. By passing to a subsequence, we may assume that  $g_k$ converges weakly to some $g\in C_b(\R^n)$; in particular $|g(0)|\ge\ve_0>0$. On the other hand, if $B_r(\xi)$ is the ball of radius $r>0$ around $\xi\in\R^n$, 
  \begin{multline*}
 \left|\int_{B_r(0)}g_k(x)\,dx\right|\le\left|\int_{B_r(\xi_k)} f(x)\,dx-\int_{B_r(\s_k)}f(x)\,dx\right|\le \|f\|_\infty \left|\left(B_r(\xi_k)\setminus B_r(\s_k)\right)\cup \left(B_r(\s_k)\setminus B_r(\xi_k)\right) \right|
  \\= \|f\|_\infty \bigl|\left(B_r(0)\setminus B_r(\xi_k-\s_k)\right)\cup \left(B_r(\xi_k-\s_k)\setminus B_r(0)\right)\bigr|\to0, \quad\text{as $k\to\infty$, for all $r>0$},
  \end{multline*}
  which gives the desired contradiction.  We also remark that, if $g\in C_b(\R^d)$ is the weak limit of a sequence of translates $f(\cdot+\l_k)$, with $f\in \WAP(\R^n)$, then 
  $g\in\BUC(\R^d)$. Indeed, weak convergence impies pointwise convergence in $\R^n$, in particular, and so, since the family $\{f(\cdot+\l_k)\}$ is equicontinuous, for 
  $f\in\BUC(\R^d)$,  it follows that $g\in\BUC(\R^d)$. 

To have a better idea of this space,  consider \v Cech compactification of $\R^d$, associated with the algebra $C_b(\R^d)$ (see, e.g., \cite{DS}), denote it by
 $\KK_0$. There is an isometric isomorphism between $C_b(\R^d)$ and $C(\KK_0)$, and weak convergence in $C_b(\R^n)$ is then translated to pointwise convergence in 
 $C(\KK_0)$. So, the weakly almost periodic functions are then identified with the functions in $C(\KK_0)$ whose sequences of translates, $(f(\cdot+\l_i))_{i\in\N}$, always possess a subsequence converging pointwise to a function $g\in C(\KK_0)$. By this characterization, it is immediate that $\WAP(\R^d)$ is an algebra in $C_b(\R^d)$, closed in the $\sup$ norm. For the following considerations on $\WAP(\R^d)$, instead of the compactification provided by all space $C_b(\R^d)$, it will be more convenient to consider the compactification provided by the algebra $\BUC(\R^d)$, which is then identified with the compact $\KK_0/\sim$ with the topology $\tau_0$ generated by the functions in 
 $\BUC(\R^d)$, where $\sim$ is the equivalence relation whose quotient makes $\tau_0$ Hausdorff. So, we have the identification of $\BUC(\R^d)$ with the space of continuous functions $C(\KK_0/\sim,\tau_0)$. In what follows we omit the quotient, writing simply $\KK_0$, instead of $\KK_0/\sim$, and will assume $\KK_0$ to be endowed with the topology 
 $\tau_0$.

  Existence of mean value for functions in $\WAP(\R^d)$ may be seen as follows.  First, by Theorem~\ref{T1}, the translations $T(y)f(\cdot)=f(\cdot+y)$ may be extended to $\KK_0$ to form a continuous dynamical system in $\KK_0$. A well known theorem by Krylov and Bogolyubov  asserts the existence of a probability measure $\mu$ in $\KK_0$, invariant by $\{T(y)\,:\,y\in\R^n\}$ (see, e.g., \cite{NS}; the extension of the proof given therein, for compact metric spaces, to general compact topological spaces is straightforward). Also,  von~Neumann mean ergodic theorem (see, e.g., \cite{DS}) implies that, given $f\in \WAP(\R^n)$, $\oM_L(f)(z):=\medint_{B_L(0)}f(T(y)z)\,dy$ converges, as $L\to\infty$, in $L^2(\KK_0,\mu)$, to a function $g(z)\in L^2(\KK_0,\mu)$ which is invariant, that is, $g(z+y)=g(z)$, for $\mu$-a.e.\ $z\in\KK_0$,  for all $y\in\R^d$. Observe that, for any $\xi\in\R^d$, 
 $$
 \Medint_{B_L(\xi)}f(T(y)\,\cdot )\,dy=\Medint_{B_L(0)}f(T(y+\xi)\,\cdot)\,dy=T(\xi)\oM_L(f)(\cdot)\to T(\xi)g(\cdot)=g(\cdot),\quad\text{as $L\to\infty$, in $L^2(\KK_0,\mu)$},
 $$ 
 by the continuity of $T(\xi):L^2(\KK_0,\mu)\to L^2(\KK_0,\mu)$, and the invariance of $g$. 
 Now, $\oM_L(f)(z)$ may be arbitrarily approximated in $C(\KK_0)$ by a finite convex combination of translates of $f$, $g_L(\cdot)=\theta_L^1 f(\cdot+\l_L^1)+\cdots+\theta_L^{K(L)}f(\cdot+\l_L^{K(L)})$, and,  taking $L=1,2,\cdots$, we may arrange that $g_L\to g$, in $L^2(\KK_0,\mu)$. Let us consider the separable closed subspace  $S\subset C(\KK_0)$ generated by the translates of $f$, $f(\cdot+\l)$, $\l\in\R^d$. The dual of $S$,  is a separable space which, by Hahn-Banach, may be viewed as a subspace of the dual of $C(\KK_0)$.
We may then define a metric $d(f,g)$ in $S$, whose induced topology is equivalent to the weak topology of $S$, and satisfies $d(f+h,g+h)=d(f,g)$.   Since the set $O(f)=\{f(\cdot+\l)\,:\,\l\in\R^d\}$ is 
pre-compact, we deduce that it is totally bounded in the metric $d$. But then, since $S$ with the weak topology is locally convex, by a well known result (see, e.g.,  \cite{Ru0}, p.72) the convex hull of $O(f)$, $\co(O(f))$, is totally bounded, and, hence, $\overline{\co(O(f))}$ is compact in the weak topology. In particular, by passing to a subsequence if necessary, 
we deduce that $g_L$ weakly converges to some $\tilde g\in C(\KK_0)$, that is $g_L(z)\to \tilde g(z)$, for all $z\in\KK_0$. But then, $\tilde g(z)=g(z)$, $\mu$-a.e., and by the invariance of $g$, we deduce that $g$ is constant  and we denote it by $\mv f$. Hence, for any $\xi\in\R^d$,  and all $z\in \KK_0$,  the averages 
  $\medint_{B_L(\xi)}f(T(y)\,z )\,dy$ converge to $\mv f$, which does not depend on either $z$ or $\xi$, and this implies that $f$ possesses mean value and this  is $\mv f$. 
  
 Taking the invariant measure $\mu$, above, as the measure induced by the mean value, we see that  the proof just given for the existence of the mean value for functions in 
 $\WAP(\R^d)$ may be repeated, line by line, to prove  the ergodicity of this algebra w.m.v., as a straightforward application of Lemma~\ref{L:1.6}. In sum, 
 $\WAP(\R^d)$ is an ergodic algebra.

 We recall that  the Fourier-Stieltjes algebra $\operatorname{FS}(\R^n)$ is defined as the closure in the $\sup$-norm of functions $f:\R^d\to\R$ which admit a representation as 
 \begin{equation}\label{ewap1}
 f(x)=\int_{\R^d}e^{ix\cdot y}\,d\mu(y),
 \end{equation}
 for some signed Radon measure in $\R^d$ with finite total variation. If $f$ admits the representation in \eqref{ewap1}, then any of its translates, $f(\cdot+\l)$, admits a similar
 representation with $\mu(y)$ replaced by $e^{i\l\cdot y}\mu(y)$. Suppose first, that $f\in\operatorname{FS}(\R^d)$, admits a representation as in \eqref{ewap1}, with $\supp\mu\subset  B_R(0)$,  for some $R>0$.   Given a sequence of translates, $f(\cdot+\l_n)$, we have that these translates satisfy an equation like
 \eqref{ewap1}, with $\mu(y)$ replaced by $\mu_n(y):=e^{\l_n\cdot y}\mu(y)$, and so $|\mu_n|(\R^d)=|\mu|(\R^d)$, and $\supp\mu_n=\supp\mu$. Since the space of Radon measures with finite total variation and support in a compact $K\subset\R^d$, $\M(K)$, is the dual of $C(K)$, we may extract a subsequence from $\mu_n$, still labeled $\mu_n$,
 such that $\mu_n\wto\nu$ in the weak-star topology of $\M(K)$, for some $\nu\in\M(K)$. Therefore, $f(\cdot+\l_n)$ pointwise converges to $g\in C_b(\R^d)$, where
 $$
 g(x)=\int_{\R^d}e^{ix\cdot y}\nu(y).
 $$
 Since any function in $\operatorname{FS}(\R^d)$ is the uniform limit of functions satisfying a representation like \eqref{ewap1}, for a signed Radon measure $\mu$ with compact support and finite total variation, we conclude that  $\operatorname{FS}(\R^d)\subset  \WAP(\R^d)$. 
\end{proof}
 
Another very important result on the ergodic algebra $\WAP(\R^d)$ was established by Eberlein in \cite{E2}. In sum, it shows that $\WAP(\R^d)\subset \APs(\R^d)$. As yet, it is not known if the latter  is an strict inclusion. 

\begin{lemma}[{\em cf.} \cite{E2}] \label{L:A.wap2}  Given any $\varphi\in\WAP(\R^d)$, then we can write,
\begin{equation}\label{e3.3.1}
\varphi=\varphi_*+\varphi_{\NN},
\end{equation}
where $\varphi_*\in\AP(\R^d)$, and $M(|\varphi_{\NN}|^2)=0$. 
\end{lemma}

 \begin{proof}  If $\K_*,\mm_*$ are the compact topological space and the invariant probability measure associated with  $\AP(\R^d)$  according to Theorem~\ref{T1},  then $\K_*$ is the so called Bohr group, which, in particular, is a topological group and $\mm_*$ coincides with the corresponding Haar measure (see, e.g., \cite{AFS}). Therefore, it is possible to define an approximation of the identity net $\{\phi_\a\}_{\a\in\Lambda}\subset\AP(\R^d)$, with $\a$  running along the directed set  $\Lambda$ of all neighborhoods of the identity, ordered naturally by $\a_1<\a_2$ if $\a_2\subset \a_1$. More specifically, for each $\a\in\Lambda$, $\phi_\a\ge0$, $\supp \phi_\a\subset \a$, and $M(\phi_\a)=1$. We may also assume that 
 $\phi_\a(-x)=\phi_\a(x)$. Therefore, given any $g\in\AP(\R^d)$, we have that $\{g_\a\}_{\a\in\Lambda}$, defined by
 $$
 g_\a(x):=g*_{{}_\oM}\phi_\a(x):=\int_{\K_*}g(x-y)\phi_\a(y)\,d\mm_*(y)
 $$
 is a net in $\AP(\R^d)$ converging uniformly to $g$. Moreover, it is easy to see that, given $\varphi\in\WAP(\R^d)$ and  $\phi\in\AP(\R^d)$, then $\phi*_{{}_\oM}\varphi$, defined by,
 $$
 \phi*_{{}_\oM}\varphi(x):=\int_{\K} \phi(x-y)\varphi(y)\,d\mm(y),
 $$
 belongs to $\AP(\R^d)$, where $\K,\mm$ are the compact topological space and the invariant probability measure associated with $\WAP(\R^d)$ by Theorem~\ref{T1}. Indeed, this follows directly from the fact that, through conveniently chosen partitions of the unity in $\K$, $\phi*_M\varphi$ maybe uniformly approximated by convex combinations of $\varphi(y_i)\phi(x-y_i)$,
 for certain finite sets of points $y_i\in\K$.  In particular, given $\varphi\in\WAP(\R^d)$, we have that the net $\varphi_\a:=\varphi*_{{}_\oM}\phi_\a$ is contained in $\AP(\R^d)$.  Now, the net $\varphi_\a$ belongs to $\overline{\co(O(\varphi))}$, the closure of the convex hull of the set of translates of $\varphi$, which, by the definition of $\WAP(\R^d)$, is compact in the weak topology (again, see, e.g.,  \cite{Ru0}, p.72). Therefore, there is a cluster point $\varphi_*$ of $\varphi_\a$. Since $\varphi_\a\in\AP(\R^d)$ and $\varphi_*$ is the weak limit of a subnet of $\varphi_\a$, we conclude that 
 $\varphi_*\in\AP(\R^d)$.

The second main observation leading to the decomposition, is the fact that any $\varphi\in\WAP(\R^d)$ has a Fourier series converging in $\BB^2$ to $\varphi$, 
that is, 
$$
\varphi(x)= \oM(\phi)+\sum_{j=1}^\infty( a_j \cos \l_j\cdot x+b_j\sin \l_j\cdot x), \quad\text{for certain $\l_j\in\R^d$, $j\in\N$},
$$
in the sense of $\BB^2$,  whose proof follows by standard arguments. In particular, $\varphi$ can be approximated in $\BB^2$ by functions in $\AP(\R^d)$. Therefore, $\varphi*\phi_\a$ converges to $\varphi$ in $\BB^2$, which implies that 
$\oM(|\varphi-\varphi_*|^2)=0$, as was to be proved.  Finally, the uniqueness of $\varphi_*$, which follows by Parseval's equation, implies that the whole net $\varphi_\a$ converges to $\varphi_*$. 

\end{proof}
 
 Finally, Rudin, in \cite{Ru}, proved that there are functions in  $\WAP(\R^d)$ which are not in  $\operatorname{FS}(\R^d)$, that is, the inclusion  $\operatorname{FS}(\R^d)\subset \WAP(\R^d)$ is strict. 

\section*{Acknowledgements}

The author gratefully acknowledges the support from CNPq, through grant proc.~303950/2009-9, and FAPERJ, through grant E-26/103.019/2011.

\end{document}